\documentclass[11pt]{amsart}
\usepackage{graphicx}
\usepackage{amscd,amsthm}
\usepackage{amsmath,color}
\usepackage{amsfonts}
\usepackage{amssymb,appendix,wasysym}
\usepackage[latin1]{inputenc}
\usepackage[colorlinks=true]{hyperref}
\usepackage{mathrsfs}
\usepackage{yfonts}
\newtheorem{theorem}{Theorem}[section]

\newtheorem{lemma}[theorem]{Lemma}

\newcommand{\be}{\begin{equation}}
\newcommand{\ee}{\end{equation}}
\numberwithin{equation}{section}

\allowdisplaybreaks

\begin{document}
\title[A fractional Kirchhoff problem]{A fractional Kirchhoff problem\\ involving a singular term and a critical nonlinearity}

\author{Alessio Fiscella}
\address{Alessio Fiscella\\
 Departamento de Matem\'atica \\ Universidade Estadual de Campinas, IMECC\\
Rua S\'ergio Buarque de Holanda, 651, Campinas, SP CEP 13083--859 Brazil}
\email{\tt fiscella@ime.unicamp.br}


\begin{abstract}
In this paper we consider the following critical nonlocal problem
$$
\left\{\begin{array}{ll}
M\left(\displaystyle\iint_{\mathbb{R}^{2N}}\frac{|u(x)-u(y)|^2}{|x-y|^{N+2s}}dxdy\right)(-\Delta)^s u = \displaystyle\frac{\lambda}{u^\gamma}+u^{2^*_s-1}&\quad\mbox{in } \Omega,\\
u>0&\quad\mbox{in } \Omega,\\
u=0&\quad\mbox{in } \mathbb{R}^N\setminus\Omega,
\end{array}\right.
$$
where $\Omega$ is an open bounded subset of $\mathbb R^N$ with continuous boundary, dimension $N>2s$ with parameter $s\in (0,1)$, $2^*_s=2N/(N-2s)$ is the fractional critical Sobolev exponent, $\lambda>0$ is a real parameter, exponent $\gamma\in(0,1)$, $M$ models a Kirchhoff type coefficient, while $(-\Delta)^s$ is the fractional Laplace operator.
In particular, we cover the delicate degenerate case, that is when the Kirchhoff function $M$ is zero at zero.
By combining variational methods with an appropriate truncation argument, we provide the existence of two solutions.
\end{abstract}

\maketitle

\section{Introduction}
This paper is devoted to the study of a class of Kirchhoff type problems driven by a nonlocal fractional operator and involving a singular term and a critical nonlinearity.
More precisely, we consider
\begin{equation}\label{P}
\left\{\begin{array}{ll}
\left(\displaystyle\iint_{\mathbb{R}^{2N}}\frac{|u(x)-u(y)|^2}{|x-y|^{N+2s}}dxdy\right)^{\theta-1} (-\Delta)^s u = \displaystyle\frac{\lambda}{u^\gamma}+u^{2^*_s-1}&\quad\mbox{in } \Omega,\\
u>0&\quad\mbox{in } \Omega,\\
u=0&\quad\mbox{in } \mathbb{R}^N\setminus\Omega,
\end{array}
\right.
\end{equation}
where $\Omega$ is an open bounded subset of $\mathbb R^N$ with continuous boundary, dimension $N>2s$ with parameter $s\in (0,1)$, $2^*_s=2N/(N-2s)$ is the fractional critical Sobolev exponent, $\lambda>0$ is a real parameter, exponent $\theta\in(1,2^*_s/2)$ while $\gamma\in(0,1)$.
Here $(-\Delta)^s$ is the fractional Laplace operator defined, up to normalization factors, by the Riesz potential as
$$(-\Delta)^s \varphi(x)=\int_{\mathbb{R}^N}\frac{2\varphi(x)-\varphi(x+y)-\varphi(x-y)}{|y|^{N+2s}}dy,\quad x\in\mathbb R^N,$$
along any $\varphi\in C^\infty_0(\Omega)$; we refer to \cite{DPV} and the recent monograph \cite{MBRS} for further details on the fractional Laplacian and the fractional Sobolev spaces $H^s(\mathbb R^N)$ and $H^s_0(\Omega)$.

As well explained in \cite{DPV}, problem \eqref{P} is the fractional version of the following nonlinear problem
\begin{equation}\label{Pc}
\left\{\begin{array}{ll}
-M\left(\int_\Omega|\nabla u(x)|^2dx\right)\Delta u = \displaystyle\frac{\lambda}{u^\gamma}+u^{2^*-1}&\quad\mbox{in } \Omega,\\
u>0&\quad\mbox{in } \Omega,\\
u=0&\quad\mbox{in } \partial\Omega,
\end{array}
\right.
\end{equation}
where $\Delta$ denotes the classical Laplace operator while, just for a general discussion, $M(t)=t^{\theta-1}$ for any $t\in\mathbb R^+_0$. In literature, problems like \eqref{P} and \eqref{Pc} are called of Kirchhoff type whenever the function $M:\mathbb R^+_0\to\mathbb R^+_0$ models the Kirchhoff prototype, given by
\begin{equation}\label{prot}
M(t)=a +bt^{\theta-1},\quad a,\,b\ge0,\,\;a+b>0,\,\;\theta\ge1.
\end{equation}
In particular, when $M(t)\ge \mbox{\textit{ constant}}>0$ for any $t\in\mathbb R^+_0$, Kirchhoff problems are said to be \textit{non--degenerate}
and this happens for example if $a>0$ in the model case \eqref{prot}.
While, if $M(0)=0$ but $M(t)>0$ for any $t\in\mathbb R^+$, Kirchhoff problems are called
\textit{degenerate}. Of course, for \eqref{prot} this occurs when $a=0$.

This kind of nonlocal problems has been widely studied in recent years. We refer to \cite{LLT, LKLT, LZLT, LS, LTLW} for different Kirchhoff problems with $M$ like in \eqref{prot}, driven by the Laplace operator and involving a singular term of type $u^{-\gamma}$. In \cite{LS}, the authors study a Kirchhoff problem with a singular term and a Hardy potential, by using the Nehari method. The same approach is used in \cite{LZLT} for a singular Kirchhoff problem with also a subcritical term. In \cite{LLT}, strongly assuming $a>0$ in \eqref{prot}, the authors prove the existence of two solutions for a Kirchhoff problem like \eqref{Pc}, by combining perturbation and variational methods. While in \cite{LKLT}, they provide a uniqueness result for a singular Kirchhoff problem involving a negative critical nonlinearity, by a minimization argument. By arguing similarly to \cite{LLT}, the authors of \cite{LTLW} give the existence of two solutions for a critical Kirchhoff problem with a singular term of type $|x|^{-\beta}u^{-\gamma}$.

Problem \eqref{P} has been studied in \cite{BIMP} when $\theta=1$, namely without a Kirchhoff coefficient. In \cite{BIMP}, they prove the existence of two solutions by applying the sub/supersolutions and Sattinger methods. In \cite{CMSS}, the authors generalize the results of Section 3 of \cite{BIMP} to the delicate case of the $p$--fractional Laplace operator $(-\Delta)^s_p$. While in the last section of \cite{AMPP}, the authors provide the existence of a solution for nonlinear fractional problems with a singularity like $u^{-\gamma}$ and a fractional Hardy term, by perturbation methods. Concerning fractional Kirchhoff problems involving critical nonlinearities, we refer to \cite{AFP, CP, FP, FP2, FV, PP} for existence results and to \cite{BFL, F, PS, PXZ, XZQ} for multiplicity results.
In particular, in \cite{CP, FP, FP2, PP} different singular terms appear, but given by the fractional Hardy potential.
 
Inspired by the above works, we study a multiplicity result for problem \eqref{P}. As far as we know, a fractional Kirchhoff problem involving a singular term of type $u^{-\gamma}$ has not been studied yet. Thus, we can state our result as follows.
\begin{theorem}\label{main} Let $s\in(0,1)$, $N>2s$, $\theta\in(1,2^*_s/2)$, $\gamma\in(0,1)$ and let $\Omega$ be an open bounded subset of $\mathbb R^N$ with $\partial\Omega$ continuous. Then, there exists $\overline{\lambda}>0$ such that for any $\lambda\in(0,\overline{\lambda})$ problem \eqref{P} has at least two different solutions.
\end{theorem}
The first solution of problem \eqref{P} is obtained by a suitable minimization argument, where we must pay attention to the nonlocal nature of the fractional Laplacian. Concerning the second solution, because of the presence of $u^{-\gamma}$, we can not apply the usual critical point theory to problem \eqref{P}. For this, we first study a perturbed problem obtained truncating the singular term $u^{-\gamma}$. Then, by approximation we get our second solution of \eqref{P}.

Finally, we observe that Theorem \ref{main} generalizes in several directions the first part of \cite[Theorem 4.1]{BIMP} and \cite[Theorem 1.1]{LLT}.

The paper is organized as follows. In Section~\ref{sec variational}, we discuss the variational formulation of problem \eqref{P} and we introduce the perturbed problem. In Section~\ref{sec existence}, we prove the existence of the first solution of \eqref{P} and we give a possible generalization of this existence result, at the end of the section.
In Section~\ref{sec mountain}, we prove the existence of a mountain pass solution for the perturbed problem.
In Section~\ref{sec finale}, we prove Theorem~\ref{main}.

\section{Variational setting}\label{sec variational}
{\em Throughout the paper we assume that $s\in(0,1)$, $N>2s$, $\theta\in(1,2^*_s/2)$, $\gamma\in(0,1)$ and $\Omega$ is an open bounded subset of $\mathbb R^N$ with $\partial\Omega$ continuous, without further mentioning. As a matter of notations, we denote with $\varphi^+=\max\{\varphi,0\}$ and $\varphi^-=\max\{-\varphi,0\}$ respectively the positive and negative part of a function $\varphi$.}

Problem~\eqref{P} has a variational structure and the natural space where finding solutions is the homogeneous fractional Sobolev space $H^s_0(\Omega)$. In order to study \eqref{P}, it is important to encode the ``boundary condition'' $u=0$ in
$\mathbb{R}^N\setminus\Omega$ in the weak formulation,
by considering also that the interaction between $\Omega$ and its complementary in $\mathbb{R}^N$ gives a positive contribution
in the so called {\em Gagliardo norm}, given as
\begin{equation}\label{norma}
\left\|u\right\|_{H^s(\mathbb R^N)}=\left\|u\right\|_{L^2(\mathbb R^N)}+\Big(\iint_{\mathbb R^{2N}} \frac{|u(x)-u(y)|^2}{\left|x-y\right|^{N+2s}}dxdy\Big)^{1/2}.
\end{equation}

The functional space that takes into account this boundary condition will be denoted by $X_0$ and it is defined as
$$
X_0=\big\{u\in H^s(\mathbb R^N):\,\,u=0\mbox{ a.e. in } \mathbb R^N\setminus \Omega\big\}.
$$
We refer to \cite{SV} for a general definition of $X_0$ and its properties.
We also would like to point out that, when $\partial\Omega$ is continuous, by \cite[Theorem~6]{FSV} the space $C^\infty_0(\Omega)$ is dense in $X_0$, with respect to the norm \eqref{norma}. This last point will be used to overcome the singularity in problem \eqref{P}.

In $X_0$ we can consider the following norm
$$
\left\|u\right\|_{X_0}=\Big(\iint_{\mathbb R^{2N}} \frac{|u(x)-u(y)|^2}{\left|x-y\right|^{N+2s}}dxdy\Big)^{1/2},
$$
which is equivalent to the usual one defined in \eqref{norma} (see \cite[Lemma~6]{SV}).
We also recall that $(X_0,\left\|\,\cdot\,\right\|_{X_0})$ is a Hilbert space, with the scalar product defined as
$$
\left\langle u,v\right\rangle_{X_0}=\iint_{\mathbb R^{2N}} \frac{(u(x)-u(y))(v(x)-v(y))}{\left|x-y\right|^{N+2s}}dxdy.
$$
{\it From now on, in order to simplify the notation, we will denote $\|\cdot\|_{X_0}$ and $\left\langle \cdot,\cdot\right\rangle_{X_0}$ by $\|\cdot\|$ and $\left\langle \cdot,\cdot\right\rangle$ respectively, and $\|\cdot\|_{L^q(\Omega)}$ by $\|\cdot\|_q$ for any $q\in[1,\infty]$.}

In order to present the weak formulation of \eqref{P} and taking into account that we are looking for positive solutions, we will consider the following Kirchhoff problem
\begin{equation}\label{P+}
\left\{\begin{array}{ll}
\left(\displaystyle\iint_{\mathbb{R}^{2N}}\frac{|u(x)-u(y)|^2}{|x-y|^{N+2s}}dxdy\right)^{\theta-1} (-\Delta)^s u = \displaystyle\frac{\lambda}{(u^+)^\gamma}+(u^+)^{2^*_s-1}&\quad\mbox{in } \Omega,\\
u=0&\quad\mbox{in } \mathbb{R}^N\setminus\Omega.
\end{array}
\right.
\end{equation}
We say that $u\in X_0$ is a (weak) solution of problem \eqref{P+}, if $u$ satisfies
\begin{equation}\label{weak}
\|u\|^{2(\theta-1)}\langle u,\varphi\rangle
 = \lambda \int_\Omega \frac{\varphi}{(u^+)^\gamma}dx
 +\int_\Omega(u^+)^{2^*_s-1}\varphi,
\end{equation}
for any $\varphi\in X_0$.
Problem \eqref{P+} has a variational structure and $J_\lambda:X_0\to\mathbb R$, defined by
\begin{align*}
J_\lambda(u)=\frac{1}{2\theta}\|u\|^{2\theta}-\frac{\lambda}{1-\gamma}\int_\Omega (u^+)^{1-\gamma}dx -\frac{1}{2^*_s}\|u^+\|^{2^*_s}_{2^*_s},
\end{align*}
is the underlying functional associated to \eqref{P+}.
Because of the presence of a singular term in \eqref{P+}, functional $J_\lambda$ is not differentiable on $X_0$. Therefore, we can not apply directly the usual critical point theory to $J_\lambda$, in order to solve problem \eqref{P+}. However, it is possible to find a first solution of \eqref{P+} by using a local minimization argument. In order to get the second solution of \eqref{P+}, we have to study an associated approximating problem. That is, for any $n\in\mathbb N$, we consider the following perturbed problem
\begin{equation}\label{Pn}
\left\{\begin{array}{ll}
\left(\displaystyle\iint_{\mathbb{R}^{2N}}\frac{|u(x)-u(y)|^2}{|x-y|^{N+2s}}dxdy\right)^{\theta-1} (-\Delta)^s u = \displaystyle\frac{\lambda}{(u^++1/n)^\gamma}+(u^+)^{2^*_s-1}\quad\mbox{in } \Omega,\\
u=0\quad\mbox{in } \mathbb{R}^N\setminus\Omega.
\end{array}
\right.
\end{equation}
For this, we say that $u\in X_0$ is a (weak) solution of problem \eqref{Pn}, if $u$ satisfies
\begin{equation}\label{weak2}
\|u\|^{2(\theta-1)}\langle u,\varphi\rangle
 = \lambda \int_\Omega \frac{\varphi}{(u^++1/n)^\gamma}dx
 +\int_\Omega(u^+)^{2^*_s-1}\varphi,
\end{equation}
for any $\varphi\in X_0$.
In this case, solutions of \eqref{Pn} correspond to the critical points of functional $J_{n,\lambda}:X_0\to\mathbb R$, set as
\begin{equation}\label{jn}
J_{n,\lambda}(u)=\frac{1}{2\theta}\|u\|^{2\theta}-\frac{\lambda}{1-\gamma}\int_\Omega[(u^++1/n)^{1-\gamma}-(1/n)^{1-\gamma}]dx -\frac{1}{2^*_s}\|u^+\|^{2^*_s}_{2^*_s}.
\end{equation}
It is immediate to see that $J_{n,\lambda}$ is of class $C^1(X_0)$.

We conclude this section recalling the best constant of the fractional Sobolev embedding, which will be very useful to study the compactness property of functional $J_{n,\lambda}$. That is, we consider
\begin{equation}\label{S}
S=\inf_{\substack{v\in H^s(\mathbb R^N)\\
v\not\equiv0}}\displaystyle\frac{\displaystyle\iint_{\mathbb{R}^{2N}}\frac{|v(x)-v(y)|^2}{|x-y|^{N+2s}}dxdy}{\left(\int_{\mathbb R^N}|v(x)|^{2^*_s}dx\right)^{2/2^*_s}},
\end{equation}
which is well defined and strictly positive, as shown in \cite[Theorem 1.1]{CT}. 

\section{A first solution for problem \eqref{P}}\label{sec existence} 
In this section we prove the existence of a solution for problem \eqref{P} by a local minimization argument. For this, we first study the geometry of functional $J_\lambda$.

\begin{lemma}\label{mp} There exist numbers $\rho\in(0,1]$, $\lambda_0=\lambda_{0}(\rho)>0$ and $\alpha=\alpha(\rho)>0$ such that $J_\lambda(u)\ge\alpha$ for any $u\in X_0$, with $\|u\|=\rho$, and for any $\lambda\in(0,\lambda_0]$.

Furthermore, set 
$$m_\lambda=\inf\left\{J_\lambda(u):\,\, u\in\overline{B}_\rho\right\},$$
where $\overline{B}_\rho=\left\{u\in X_0:\,\,\|u\|\leq\rho\right\}$. Then, $m_\lambda<0$ for any $\lambda\in(0,\lambda_0]$.
\end{lemma}
\begin{proof}
Let $\lambda>0$. From the H\"older inequality and \eqref{S}, for any $u\in X_0$ we have 
\begin{equation}\label{holder}
\int_\Omega(u^+)^{1-\gamma}dx\leq|\Omega|^{\frac{2^*_s-1+\gamma}{2^*_s}}\|u\|^{1-\gamma}_{2^*_s}\leq|\Omega|^{\frac{2^*_s-1+\gamma}{2^*_s}}S^{-\frac{1-\gamma}{2}}\|u\|^{1-\gamma}.
\end{equation}
Hence, by using again \eqref{S} and \eqref{holder} we get
$$
J_\lambda(u)\ge\frac{1}{2\theta}\|u\|^{2\theta}-\frac{S^{-\frac{2^*_s}{2}}}{2^*_s}\|u\|^{2^*_s}-\frac{\lambda}{1-\gamma}|\Omega|^{\frac{2^*_s-1+\gamma}{2^*_s}}S^{-\frac{1-\gamma}{2}}\|u\|^{1-\gamma}.
$$
Since $1-\gamma<1<2\theta<2^*_s$, the function
$$\eta(t)=\frac{1}{2\theta}t^{2\theta-1+\gamma}-\frac{S^{-\frac{2^*_s}{2}}}{2^*_s}t^{2^*_s-1+\gamma},\quad t\in[0,1]$$
admits a maximum at some $\rho\in(0,1]$ small enough, that is $\displaystyle\max_{t\in[0,1]}\eta(t)=\eta(\rho)>0$.
Thus, let
$$\lambda_0=\frac{(1-\gamma)S^{\frac{1-\gamma}{2}}}{2|\Omega|^{\frac{2^*_s-1+\gamma}{2^*_s}}}\eta(\rho),
$$
then for any $u\in X_0$ with $\|u\|=\rho\leq1$ and for any $\lambda\leq\lambda_0$, we get $J_\lambda(u)\ge\rho^{1-\gamma}\eta(\rho)/2=\alpha>0$. 

Furthermore, fixed $v\in X_0$ with $v^+\not\equiv0$, for $t\in(0,1)$ sufficiently small
$$
J_\lambda(tv)=\frac{t^{2\theta}}{2\theta}\|v\|^{2\theta}-t^{1-\gamma}\frac{\lambda}{1-\gamma}\int_\Omega (v^+)^{1-\gamma}dx -\frac{t^{2^*_s}}{2^*_s}\|v^+\|^{2^*_s}_{2^*_s}<0,
$$
being $1-\gamma<1<2\theta<2^*_s$. This concludes the proof.
\end{proof}

We are now ready to prove the existence of the first solution of \eqref{P}.

\begin{theorem}\label{prima} Let $\lambda_0$ be given as in Lemma \ref{mp}. Then, for any $\lambda\in(0,\lambda_0]$ problem \eqref{P} has a solution $u_0\in X_0$, with $J_\lambda(u_0)<0$.
\end{theorem}
\begin{proof}
Fix $\lambda\in(0,\lambda_0]$ and let $\rho$ be as given in Lemma \ref{mp}. We first prove there exists $u_0\in \overline{B}_\rho$ such that $J_\lambda(u_0)=m_\lambda<0$.
Let $\{u_k\}_k\subset\overline{B}_\rho$ be a minimizing sequence for $m_\lambda$, that is such that 
\begin{equation}\label{minimax}
\lim_{k\to\infty}J_\lambda(u_k)=m_\lambda.
\end{equation}
Since $\{u_k\}_k$ is bounded in $X_0$, by applying \cite[Lemma~8]{SV} and \cite[Theorem 4.9]{B}, there exist a subsequence, still
denoted by $\{u_k\}_k$, and a function $u_0\in\overline{B}_\rho$ such that, as $k\to\infty$ we have
\begin{equation}\label{convergences}
\begin{array}{ll}
u_k\rightharpoonup u_0\text{ in }X_0,\quad &u_k\rightharpoonup u_0\text{ in }L^{2^*_s}(\Omega), \\
u_k\to u_0\text{ in }L^p(\Omega)\mbox{ for any }p\in[1,2^*_s),\quad &u_k\to u_0\text{ a.e.  in }\Omega.
\end{array}
\end{equation}
Since $\gamma\in(0,1)$, by the H\"older inequality, for any $k\in\mathbb N$ we have
$$\left|\int_\Omega (u^+_k)^{1-\gamma}dx-\int_\Omega (u^+_0)^{1-\gamma}dx\right|\leq\int_\Omega\left|u^+_k-u^+_0\right|^{1-\gamma}dx
\leq\|u^+_k-u^+_0\|^{1-\gamma}_2|\Omega|^{\frac{1+\gamma}{2}},
$$
which yields, by \eqref{convergences}
\begin{equation}\label{gamma}
\lim_{k\to\infty}\int_\Omega (u^+_k)^{1-\gamma}dx=\int_\Omega (u^+_0)^{1-\gamma}dx.
\end{equation}
Let $w_k=u_k-u_0$, by \cite[Theorem 2]{BL} it holds true that
\begin{equation}\label{bl}
\|u_k\|^2=\|w_k\|^2+\|u_0\|^2+o(1),\quad\|u_k\|^{2^*_s}_{2^*_s}=\|w_k\|^{2^*_s}_{2^*_s}+\|u_0\|^{2^*_s}_{2^*_s}+o(1)
\end{equation}
as $k\to\infty$.
Since $\{u_k\}_k\subset\overline{B}_\rho$, by \eqref{bl} for $k$ sufficiently large, we have $w_k\in\overline{B}_\rho$.
Lemma \ref{mp} implies that for any $u\in X_0$, with $\|u\|=\rho$, we get
$$\frac{1}{2\theta}\|u\|^{2\theta}-\frac{1}{2^*_s}\|u^+\|^{2^*_s}_{2^*_s}\geq\alpha>0,
$$
and from this, being $\rho\leq1$, for $k$ sufficiently large we have
\begin{equation}\label{2.10}
\frac{1}{2\theta}\|w_k\|^{2\theta}-\frac{1}{2^*_s}\|w^+_k\|^{2^*_s}_{2^*_s}>0.
\end{equation}
Thus, by \eqref{minimax}, \eqref{gamma}--\eqref{2.10} and considering $\theta\geq1$, it follows that as $k\to\infty$
\begin{align*}
m_\lambda&=J_\lambda(u_k)+o(1)\\
&=\frac{1}{2\theta}\left(\|w_k\|^2+\|u_0\|^2\right)^{\theta}-\frac{\lambda}{1-\gamma}\int_\Omega (u^+_0)^{1-\gamma}dx -\frac{1}{2^*_s}\left(\|w^+_k\|^{2^*_s}_{2^*_s}+\|u^+_0\|^{2^*_s}_{2^*_s}\right)+o(1)\\
&\geq J_\lambda(u_0)+\frac{1}{2\theta}\|w_k\|^{2\theta}-\frac{1}{2^*_s}\|w^+_k\|^{2^*_s}_{2^*_s}+o(1)\geq J_\lambda(u_0)+o(1)\geq m_\lambda,
\end{align*}
being $u_0\in\overline{B}_\rho$.
Hence, $u_0$ is a local minimizer for $J_\lambda$, with $J_\lambda(u_0)=m_\lambda<0$ which implies that $u_0$ is nontrivial.

Now, we prove that $u_0$ is a positive solution of \eqref{P+}. For any $\psi\in X_0$, with $\psi\geq0$ a.e. in $\mathbb R^N$, let us consider a $t>0$ sufficiently small so that $u_0+t\psi\in\overline{B}_\rho$. Since $u_0$ is a local minimizer for $J_\lambda$, we have
\begin{align*}
0&\leq J_\lambda(u_0+t\psi)-J_\lambda(u_0)\\
&=\frac{1}{2\theta}\left(\|u_0+t\psi\|^{2\theta}-\|u_0\|^{2\theta}\right)-\frac{\lambda}{1-\gamma}\int_\Omega \left[((u_0+t\psi)^+)^{1-\gamma}-(u^+_0)^{1-\gamma}\right]dx\\
&\quad-\frac{1}{2^*_s}\left(\|u_0+t\psi\|^{2^*_s}_{2^*_s}-\|u^+_0\|^{2^*_s}_{2^*_s}\right).
\end{align*}
From this, dividing by $t>0$ and passing to the limit as $t\to0^+$, it follows that
\begin{equation}\label{2.12}
\liminf_{t\to0^+}\frac{\lambda}{1-\gamma}\int_\Omega\frac{((u_0+t\psi)^+)^{1-\gamma}-(u^+_0)^{1-\gamma}}{t}dx\leq\|u_0\|^{2(\theta-1)}\langle u_0,\psi\rangle-\int_\Omega(u^+_0)^{2^*_s-1}\psi dx.
\end{equation}
We observe that
$$
\frac{1}{1-\gamma}\!\cdot\!\frac{((u_0+t\psi)^+)^{1-\gamma}-(u^+_0)^{1-\gamma}}{t}=((u_0+\xi t\psi)^+)^{-\gamma}\psi\quad\mbox{ a.e. in }\Omega,
$$
with $\xi\in(0,1)$ and $((u_0+\xi t\psi)^+)^{-\gamma}\to(u^+_0)^{-\gamma}\psi$ a.e. in $\Omega$, as $t\to0^+$.
Thus, by the Fatou lemma, we obtain
\begin{equation}\label{2.12b}
\lambda\int_\Omega(u^+_0)^{-\gamma}\psi dx\leq\liminf_{t\to0^+}\frac{\lambda}{1-\gamma}\int_\Omega\frac{((u_0+t\psi)^+)^{1-\gamma}-(u^+_0)^{1-\gamma}}{t}dx.
\end{equation}
Therefore, combining \eqref{2.12} and \eqref{2.12b} we get
\begin{equation}\label{2.13}
\|u_0\|^{2(\theta-1)}\langle u_0,\psi\rangle-\lambda\int_\Omega(u^+_0)^{-\gamma}\psi dx-\int_\Omega(u^+_0)^{2^*_s-1}\psi dx\geq0,
\end{equation}
for any $\psi\in X_0$, with $\psi\geq0$ a.e. in $\mathbb R^N$.

Since $J_\lambda(u_0)<0$ and by Lemma \ref{mp}, we have $u_0\in B_\rho$. Hence, there exists $\delta\in(0,1)$ such that $(1+t)u_0\in\overline{B}_\rho$ for any $t\in[-\delta,\delta]$. Let us define $I_\lambda(t)=J_\lambda((1+t)u_0)$. Since $u_0$ is a local minimizer for $J_\lambda$ in $\overline{B}_\rho$, functional $I_\lambda$ has a minimum at $t=0$, that is
\begin{equation}\label{2.14}
I'_\lambda(0)=\|u_0\|^{2\theta}-\lambda\int_\Omega(u^+_0)^{1-\gamma}dx-\int_\Omega(u^+_0)^{2^*_s}dx=0.
\end{equation}
For any $\varphi\in X_0$ and any $\varepsilon>0$, let us define $\psi_\varepsilon=u^+_0+\varepsilon\varphi$. Then, by \eqref{2.13} we have
\begin{equation}\label{2.15a}
\begin{aligned}
0&\leq\|u_0\|^{2(\theta-1)}\langle u_0,\psi^+_\varepsilon\rangle-\lambda\int_\Omega(u^+_0)^{-\gamma}\psi^+_\varepsilon dx-\int_\Omega(u^+_0)^{2^*_s-1}\psi^+_\varepsilon dx\\
&=\|u_0\|^{2(\theta-1)}\langle u_0,\psi_\varepsilon+\psi^-_\varepsilon\rangle-\lambda\int_\Omega(u^+_0)^{-\gamma}(\psi_\varepsilon+\psi^-_\varepsilon) dx-\int_\Omega(u^+_0)^{2^*_s-1}(\psi_\varepsilon+\psi^-_\varepsilon) dx.
\end{aligned}
\end{equation}
We observe that, for a.e. $x$, $y\in\mathbb{R}^N$, we obtain
\begin{equation}\label{eur}
\begin{aligned}
(u_0(x)&-u_0(y))(u^-_0(x)-u^-_0(y))\\
&=-u^+_0(x)u^-_0(y)-u^-_0(x)u^+_0(y)-\big[u^-_0(x)-u^-_0(y)\big]^2\\
&\leq-\left|u^-_0(x)-u^-_0(y)\right|^2,
\end{aligned}
\end{equation}
from which we immediately get
$$
(u_0(x)-u_0(y))(u^+_0(x)-u^+_0(y))\leq\left|u_0(x)-u_0(y)\right|^2.
$$
From the last inequality, it follows that
\begin{equation}\label{2.15b}
\begin{aligned}
\langle u_0,\psi_\varepsilon+\psi^-_\varepsilon\rangle&=\iint_{\mathbb R^{2N}}\frac{(u_0(x)-u_0(y))(\psi_\varepsilon(x)+\psi^-_\varepsilon(x)-\psi_\varepsilon(y)-\psi^-_\varepsilon(y))}{|x-y|^{N+2s}}dxdy\\
&\leq\iint_{\mathbb R^{2N}}\frac{|u_0(x)-u_0(y)|^2}{|x-y|^{N+2s}}dxdy+\varepsilon\iint_{\mathbb R^{2N}}\frac{(u_0(x)-u_0(y))(\varphi(x)-\varphi(y))}{|x-y|^{N+2s}}dxdy\\
&\quad+\iint_{\mathbb R^{2N}}\frac{(u_0(x)-u_0(y))(\psi^-_\varepsilon(x)-\psi^-_\varepsilon(y))}{|x-y|^{N+2s}}dxdy.
\end{aligned}
\end{equation}
Hence, denoting with $\Omega_\varepsilon=\left\{x\in\mathbb R^N:\,\,u^+_0(x)+\varepsilon\varphi(x)\leq0\right\}$ and by combining \eqref{2.15a} with \eqref{2.15b}, we get
\begin{equation}\label{2.15c}
\begin{aligned}
0&\leq\|u_0\|^{2\theta}+\varepsilon\|u_0\|^{2(\theta-1)}\langle u_0,\varphi\rangle+\|u_0\|^{2(\theta-1)}\langle u_0,\psi^-_\varepsilon\rangle\\
&\quad-\lambda\int_\Omega(u^+_0)^{-\gamma}(u^+_0+\varepsilon\varphi) dx-\int_\Omega(u^+_0)^{2^*_s-1}(u^+_0+\varepsilon\varphi) dx\\
&\quad+\lambda\int_{\Omega_\varepsilon}(u^+_0)^{-\gamma}(u^+_0+\varepsilon\varphi)dx+\int_{\Omega_\varepsilon}(u^+_0)^{2^*_s-1}(u^+_0+\varepsilon\varphi) dx\\
&\leq\|u_0\|^{2\theta}-\lambda\int_\Omega(u^+_0)^{1-\gamma}dx-\int_\Omega(u^+_0)^{2^*_s}dx+\|u_0\|^{2(\theta-1)}\langle u_0,\psi^-_\varepsilon\rangle\\
&\quad+\varepsilon\left[\|u_0\|^{2(\theta-1)}\langle u_0,\varphi\rangle-\lambda\int_\Omega(u^+_0)^{-\gamma}\varphi dx-\int_\Omega(u^+_0)^{2^*_s-1}\varphi dx\right]\\
&=\|u_0\|^{2(\theta-1)}\langle u_0,\psi^-_\varepsilon\rangle+\varepsilon\left[\|u_0\|^{2(\theta-1)}\langle u_0,\varphi\rangle-\lambda\int_\Omega(u^+_0)^{-\gamma}\varphi dx-\int_\Omega(u^+_0)^{2^*_s-1}\varphi dx\right],
\end{aligned}
\end{equation}
where last equality follows from \eqref{2.14}.
Arguing similarly to \eqref{eur}, for a.e. $x$, $y\in\mathbb{R}^N$ we have
\begin{equation}\label{eur2}
(u_0(x)-u_0(y))(u^+_0(x)-u^+_0(y))\geq\left|u^+_0(x)-u^+_0(y)\right|^2.
\end{equation}
Thus, denoting with
$$
\mathcal U_\varepsilon(x,y)=\frac{(u_0(x)-u_0(y))(\psi^-_\varepsilon(x)-\psi^-_\varepsilon(y))}{|x-y|^{N+2s}},
$$
by the symmetry of the fractional kernel and \eqref{eur2}, we get
\begin{equation}\label{2.15d}
\begin{aligned}
\langle u_0,\psi^-_\varepsilon\rangle&=\iint_{\Omega_\varepsilon\times\Omega_\varepsilon}\mathcal U_\varepsilon(x,y)dxdy+2\iint_{\Omega_\varepsilon\times(\mathbb R^N\setminus\Omega_\varepsilon)}\mathcal U_\varepsilon(x,y)dxdy\\
&\leq-\varepsilon\left(\iint_{\Omega_\varepsilon\times\Omega_\varepsilon}\mathcal U(x,y)dxdy+2\iint_{\Omega_\varepsilon\times(\mathbb R^N\setminus\Omega_\varepsilon)}\mathcal U(x,y)dxdy\right)\\
&\leq2\varepsilon\iint_{\Omega_\varepsilon\times\mathbb R^N}\left|\mathcal U(x,y)\right|dxdy,
\end{aligned}
\end{equation}
where we set
$$
\mathcal U(x,y)=\frac{(u_0(x)-u_0(y))(\varphi(x)-\varphi(y))}{|x-y|^{N+2s}}.
$$
Clearly $\mathcal U\in L^1(\mathbb R^{2N})$, so that for any $\sigma>0$ there exists $R_\sigma$ sufficiently large such that
$$\iint_{(\mbox{\small supp }\varphi)\times(\mathbb R^N\setminus B_{R_\sigma})}\left|\mathcal U(x,y)\right|dxdy<\frac{\sigma}{2}.
$$
Also, by definition of $\Omega_\varepsilon$, we have $\Omega_\varepsilon\subset\mbox{\small supp }\varphi$ and $|\Omega_\varepsilon\times B_{R_\sigma}|\to0$ as $\varepsilon\to0^+$. Thus, since $\mathcal U\in L^1(\mathbb R^{2N})$, there exists $\delta_\sigma>0$ and $\varepsilon_\sigma>0$ such that for any $\varepsilon\in(0,\varepsilon_\sigma]$
$$|\Omega_\varepsilon\times B_{R_\sigma}|<\delta_\sigma\quad\mbox{and}\quad\iint_{\Omega_\varepsilon\times B_{R_\sigma}}\left|\mathcal U(x,y)\right|dxdy<\frac{\sigma}{2}.
$$
Therefore, for any $\varepsilon\in(0,\varepsilon_\sigma]$
$$\iint_{\Omega_\varepsilon\times\mathbb R^N}\left|\mathcal U(x,y)\right|dxdy<\sigma,
$$
from which we get
\begin{equation}\label{2.15e}
\lim_{\varepsilon\to0^+}\iint_{\Omega_\varepsilon\times\mathbb R^N}\left|\mathcal U(x,y)\right|dxdy=0.
\end{equation}
Combining \eqref{2.15c} with \eqref{2.15d}, dividing by $\varepsilon$, letting $\varepsilon\to0^+$ and using \eqref{2.15e}, we obtain
$$
\|u_0\|^{2(\theta-1)}\langle u_0,\varphi\rangle-\lambda\int_\Omega(u^+_0)^{-\gamma}\varphi dx-\int_\Omega(u^+_0)^{2^*_s-1}\varphi dx\geq0,
$$
for any $\varphi\in X_0$. By the arbitrariness of $\varphi$, we prove that $u_0$ verifies \eqref{weak}, that is $u_0$ is a nontrivial solution of \eqref{P+}.

Finally, by considering $\varphi=u^-_0$ in \eqref{weak} and using \eqref{eur}, we see that $\|u^-_0\|=0$, which implies that $u_0$ is nonnegative. Moreover, by the maximum principle in \cite[Proposition 2.2.8]{S}, we can deduce that $u_0$ is a positive solution of \eqref{P+} and so also solves problem \eqref{P}. This concludes the proof.
\end{proof}

We end this section observing that the result in Theorem \ref{prima} can be extended to more general Kirchhoff problems.
That is, we can consider the following problem
\begin{equation}\label{Pm}
\left\{\begin{array}{ll}
M\left(\displaystyle\iint_{\mathbb{R}^{2N}}\frac{|u(x)-u(y)|^p}{|x-y|^{N+ps}}dxdy\right)(-\Delta)^s_p u = \displaystyle\frac{\lambda}{u^\gamma}+u^{p^*_s-1}&\quad\mbox{in } \Omega,\\
u>0&\quad\mbox{in } \Omega,\\
u=0&\quad\mbox{in } \mathbb{R}^N\setminus\Omega,
\end{array}
\right.
\end{equation}
where $p^*_s=pN/(N-ps)$, with here $N>ps$ and $p>1$, while the Kirchhoff coefficient $M$ satisfies condition
\begin{enumerate}
\item[$(\mathcal M)$]
{\em $M:\mathbb R^+_0\rightarrow\mathbb R^+_0$ is continuous and nondecreasing. There exist numbers $a>0$ and $\vartheta$ such that for any $t\in\mathbb R^+_0$
$$
\mathscr M(t):=\int_0^t M(\tau)d\tau\geq a\,t^\vartheta,\quad\mbox{with }\vartheta
\begin{cases}
\in(1,p^*_s/p) & \mbox{ if }M(0)=0,\\
=1 & \mbox{ if }M(0)>0.
\end{cases}
$$}\end{enumerate}
The main operator $(-\Delta)^s_p$ is the fractional $p$--Laplacian which may be defined, for any function $\varphi\in C^\infty_0(\Omega)$, as
\begin{equation*}
(-\Delta)^s_p \varphi(x)= 2\lim_{\varepsilon\rightarrow 0^+}\int_{\mathbb{R}^N\setminus B_\varepsilon(x)} \frac{|\varphi(x)-\varphi(y)|^{p-2}\big(\varphi(x)-\varphi(y)\big)}{|x-y|^{N+ps}}dy,\quad x\in\mathbb R^N,
\end{equation*}
where $B_\varepsilon(x)=\{y\in\mathbb{R}^N:\,\,|x-y|<\varepsilon\}$.
Then, arguing as in the proof of Theorem \ref{prima} and observing that we have not used yet the assumption that $\partial\Omega$ is continuous, we can prove the following result.

\begin{theorem} Let $s\in(0,1)$, $p>1$, $N>ps$, $\gamma\in(0,1)$ and let $\Omega$ be an open bounded subset of $\mathbb R^N$. Let $M$ satisfy $(\mathcal M)$. Then, there exists $\lambda_0>0$ such that for any $\lambda\in(0,\lambda_0]$ problem \eqref{Pm} admits a solution.
\end{theorem}

\section{A mountain pass solution for problem \eqref{Pn}}\label{sec mountain}

In this section we prove the existence of a positive solution for perturbed problem \eqref{Pn}, by the mountain pass theorem. {\em For this, throughout this section we assume $n\in\mathbb N$, without further mentioning}. Now, we first prove that the related functional $J_{n,\lambda}$ satisfies all the geometric features required by the mountain pass theorem.

\begin{lemma}\label{mp2} Let $\rho\in(0,1]$, $\lambda_0=\lambda_{0}(\rho)>0$ and $\alpha=\alpha(\rho)>0$ be given as in Lemma \ref{mp}.
Then, for any $\lambda\in(0,\lambda_0]$ and any $u\in X_0$, with $\|u\|\leq \rho$, functional $J_{n,\lambda}(u)\ge\alpha$.

Furthermore, there exists $e\in X_0$, with $\|e\|>\rho$, such that $J_{n,\lambda}(e)<0$.
\end{lemma}
\begin{proof}
Since $\gamma\in(0,1)$, by the subadditivity of $t\mapsto t^{1-\gamma}$, we have
\begin{equation}\label{subadd}
(u^++1/n)^{1-\gamma}-(1/n)^{1-\gamma}\leq(u^+)^{1-\gamma}\quad\mbox{a.e. in }\Omega, 
\end{equation}
for any $u\in X_0$ and any $n\in\mathbb N$. Thus, we have $J_{n,\lambda}(u)\ge J_\lambda(u)$ for any $u\in X_0$ and the first part of lemma directly follows by Lemma \ref{mp}.

For any $v\in X_0$, with $v^+\not\equiv0$, and $t>0$ we have
$$
J_{n,\lambda}(tv)=\frac{t^{2\theta}}{2\theta}\|v\|^{2\theta}-\frac{\lambda}{1-\gamma}\int_\Omega\left[(tv^++1/n)^{1-\gamma}-(1/n)^{1-\gamma}\right]dx -\frac{t^{2^*_s}}{2^*_s}\|v^+\|^{2^*_s}_{2^*_s}\to-\infty\quad\mbox{as }t\to\infty,
$$
being $1-\gamma<1<2\theta<2^*_s$. Hence, we can find $e\in X_0$, with $\|e\|>\rho$ sufficiently large, such that $J_{n,\lambda}(e)<0$. This concludes the proof.
\end{proof}

We discuss now the compactness property for the functional $J_{n,\lambda}$, given by
the Palais--Smale condition. We recall that $\{u_k\}_k\subset X_0$
is a Palais--Smale sequence for $J_{n,\lambda}$ at level $c\in\mathbb R$ if
\begin{equation}\label{e2.1}
J_{n,\lambda}(u_k)\to c\quad\mbox{and}\quad J'_{n,\lambda}(u_k)\to 0\quad\mbox{in $(X_0)'$ as }k\to\infty.
\end{equation}
We say that $J_{n,\lambda}$ satisfies the Palais--Smale condition at level $c$ if
any Palais--Smale sequence $\{u_k\}_k$ at level $c$ admits a convergent subsequence in $X_0$.

Before proving this compactness condition, we introduce the following positive
constants which will help us for a better explanation
\begin{equation}\label{costanti}
D_1=\left(\frac{1}{2\theta}-\frac{1}{2^*_s}\right)S^{\frac{2^*_s\theta}{2^*_s-2\theta}}\qquad D_2=\frac{\displaystyle\left[\left(\frac{1}{1-\gamma}+\frac{1}{2^*_s}\right)|\Omega|^{\frac{2^*_s-1+\gamma}{2^*_s}}S^{-\frac{1-\gamma}{2}}\right]^{\frac{2\theta}{2\theta-1+\gamma}}}{\displaystyle\left(\frac{1}{2\theta}-\frac{1}{2^*_s}\right)^{\frac{1-\gamma}{2\theta-1+\gamma}}}.
\end{equation}

\begin{lemma}\label{palais} Let $\lambda>0$. Then, the functional $J_{n,\lambda}$ satisfies the Palais--Smale condition at any level $c\in\mathbb R$ verifying
\begin{equation}\label{livello}
c<D_1-D_2\lambda^{2\theta/(2\theta-1+\gamma)},
\end{equation}
with $D_1$, $D_2>0$ given as in \eqref{costanti}.
\end{lemma}

\begin{proof}
Let $\lambda>0$ and let $\{u_k\}_k$ be a Palais--Smale sequence in $X_0$ at level $c\in\mathbb R$, with $c$ satisfying \eqref{livello}. We first prove the boundedness of $\{u_k\}_k$. By \eqref{e2.1} there exists $\sigma>0$ such that, as $k\to\infty$
$$
\begin{aligned}
c+\sigma\|u_k\|+o(1)&\geq J_{n,\lambda}(u_k)-\frac{1}{2^*_s}\langle J'_{n,\lambda}(u_k), u_k\rangle\\
&=\left(\frac{1}{2\theta}-\frac{1}{2^*_s}\right)\|u_k\|^{2\theta}-\frac{\lambda}{1-\gamma}\int_\Omega[(u^+_k+1/n)^{1-\gamma}-(1/n)^{1-\gamma}]dx\\
&\quad+\frac{\lambda}{2^*_s}\int_\Omega(u^+_k+1/n)^{-\gamma}u_kdx\\
&\geq\left(\frac{1}{2\theta}-\frac{1}{2^*_s}\right)\|u_k\|^{2\theta}-\lambda\left(\frac{1}{1-\gamma}+\frac{1}{2^*_s}\right)\int_\Omega|u_k|^{1-\gamma}dx\\
&\geq\left(\frac{1}{2\theta}-\frac{1}{2^*_s}\right)\|u_k\|^{2\theta}-\lambda\left(\frac{1}{1-\gamma}+\frac{1}{2^*_s}\right)|\Omega|^{\frac{2^*_s-1+\gamma}{2^*_s}}S^{-\frac{1-\gamma}{2}}\|u_k\|^{1-\gamma}
\end{aligned}
$$
where the last two inequalities follow by \eqref{S}, \eqref{subadd} and the H\"older inequality.
Therefore, $\{u_k\}_k$ is bounded in $X_0$, being $1-\gamma<1<2\theta$.
Also, $\{u^-_k\}_k$ is bounded in $X_0$ and by \eqref{e2.1} we have
$$
\lim_{k\to\infty}\langle J'_{n,\lambda}(u_k), -u^-_k\rangle=\lim_{k\to\infty}\|u_k\|^{2(\theta-1)}\langle u_k,-u^-_k\rangle=0.
$$
Thus, by inequality \eqref{eur} we deduce that $\|u^-_k\|\to0$ as $k\to\infty$, which yields
$$
J_{n,\lambda}(u_k)=J_{n,\lambda}(u_k^+)+o(1)\quad\mbox{and}\quad J'_{n,\lambda}(u_k)=J'_{n,\lambda}(u_k^+)+o(1)\quad\mbox{as }k\to\infty.
$$
Hence, we can suppose that $\{u_k\}_k$ is a sequence of nonnegative functions. 

By the boundedness of $\{u_k\}_k$ and by using \cite[Lemma 8]{SV} and \cite[Theorem 4.9]{B}, there exist a subsequence, still
denoted by $\{u_k\}_k$, and a function $u\in X_0$
such that
\begin{equation}\label{e2.4}
\begin{array}{ll}
u_k\rightharpoonup u\text{ in }X_0,\quad & \|u_k\|\rightarrow \mu, \\
u_k\rightharpoonup u\mbox{ in } L^{2^*_s}(\Omega),\quad &\left\|u_k-u\right\|_{2^*_s}\to\ell,\\
u_k\rightarrow u\mbox{ in } L^p(\Omega)\text{ for any }p\in[1,2^*_s),\quad&u_k\to u\text{ a.e.  in }\Omega,\quad u_k\leq h\text{ a.e.  in }\Omega,
\end{array}
\end{equation}
as $k\to\infty$, with $h\in L^p(\Omega)$ for a fixed $p\in[1,2^*_s)$.
If $\mu=0$, then immediately $u_k\to0$ in $X_0$ as $k\to\infty$. Hence, let us assume that $\mu>0$.

Since $n\in\mathbb N$, by \eqref{e2.4} it follows that
$$\left|\frac{u_k-u}{(u_k+1/n)^\gamma}\right|\leq n^\gamma(h+|u|)\quad\mbox{a.e. in }\Omega,
$$
so by the dominated convergence theorem and \eqref{e2.4}, we have
\begin{equation}\label{e2.5}
\lim_{k\to\infty}\int_{\Omega}\frac{u_k-u}{(u_k+1/n)^\gamma}dx=0.
\end{equation}
By \eqref{e2.4} and \cite[Theorem 2]{BL} we have
\begin{align}\label{e2.9}
\|u_k\|^2=\|u_k-u\|^2+\|u\|^2+o(1),\quad
\|u_k\|_{2^*_s}^{2^*_s}=\|u_k-u\|_{2^*_s}^{2^*_s}+\|u\|_{2^*_s}^{2^*_s}+o(1)
\end{align}
as $k\to\infty$.
Consequently, we deduce from
\eqref{e2.1}, \eqref{e2.4}, \eqref{e2.5} and \eqref{e2.9} that, as $k\to\infty$
\begin{align*}
o(1)&=\langle J_{n,\lambda}^\prime(u_k),u_k-u\rangle
\nonumber\\
&=\|u_k\|^{2(\theta-1)}\langle u_k,u_k-u\rangle
-\lambda\int_{\Omega}\frac{u_k-u}{(u_k+1/n)^\gamma}dx
-\int_{\Omega}u_k^{2^*_s-1}(u_k-u)dx\\
&=\mu^{2(\theta-1)}(\mu^2-\|u_k\|^2)
-\|u_k\|^{2^*_s}_{2^*_s}+\|u\|^{2^*_s}_{2^*_s}
+o(1)=\mu^{2(\theta-1)}\|u_k-u\|^2
-\|u_k-u\|^{2^*_s}_{2^*_s} +o(1).\nonumber
\end{align*}
Therefore, we have proved the crucial formula
\begin{equation}\label{I}
\mu^{2(\theta-1)}\lim_{k\to\infty}\|u_k-u\|^2=\lim_{k\to\infty}\|u_k-u\|^{2^*_s}_{2^*_s}.
\end{equation}
If $\ell=0$, since $\mu>0$, by \eqref{e2.4} and \eqref{I} we have $u_k\to u$ in $X_0$ as $k\to\infty$, concluding the proof.

Thus, let us assume by contradiction that $\ell>0$.
By \eqref{S}, the notation in \eqref{e2.4} and \eqref{I}, we get
\begin{equation}\label{ll}
\ell^{2^*_s}\ge S\, \mu^{2(\theta-1)}\ell^2.
\end{equation}
Noting that \eqref{I} implies in particular that
$$\mu^{2(\theta-1)}\big(\mu^2-\|u\|^2\big)=\ell^{2^*_s},$$
using \eqref{ll}, it follows that
$$
\big(\ell^{2^*_s}\big)^{2s/N}=(\mu^{2(\theta-1)})^{2s/N}\big(\mu^2-\|u\|^2\big)^{2s/N}\ge S\,\mu^{2(\theta-1)}.
$$
From this, we obtain
$$\mu^{4s/N}\ge\big(\mu^2-\|u\|^2\big)^{2s/N}\ge S\, (\mu^{2(\theta-1)})^{\frac{N-2s}{N}},$$
and considering $N<2s\theta/(\theta-1)=2s\theta'$, we have
\begin{align}\label{eq31}
\mu^2\geq S^{\frac{N}{2s\theta-N(\theta-1)}}.
\end{align}
Indeed, the restriction $N/(2\theta')<s$ follows directly from the fact that $1<\theta<2^*_s/2=N/(N-2s)$.
By \eqref{subadd} and considering that $n\in\mathbb N$, for any $k\in\mathbb N$ we have
$$
J_{n,\lambda}(u_k)-\frac{1}{2^*_s}\left\langle J'_{n,\lambda}(u_k),u_k\right\rangle\geq\left(\frac{1}{2\theta}-\frac{1}{2^*_s}\right)\|u_k\|^{2\theta}-\lambda\left(\frac{1}{1-\gamma}+\frac{1}{2^*_s}\right)\int_\Omega u_k^{1-\gamma}dx.
$$
From this, as $k\to\infty$, being $\theta\geq1$, by \eqref{e2.1}, \eqref{e2.4}, \eqref{e2.9}, \eqref{eq31}, the H\"older inequality and the Young inequality, we obtain
$$
\begin{alignedat}4
c&\geq\left(\frac{1}{2\theta}-\frac{1}{2^*_s}\right)\left(\mu^{2\theta}+\|u\|^{2\theta}\right)-\lambda\left(\frac{1}{1-\gamma}+\frac{1}{2^*_s}\right)|\Omega|^{\frac{2^*_s-1+\gamma}{2^*_s}}S^{-\frac{1-\gamma}{2}}\|u\|^{1-\gamma}\\
&\geq\left(\frac{1}{2\theta}-\frac{1}{2^*_s}\right)\left(\mu^{2\theta}+\|u\|^{2\theta}\right)-\left(\frac{1}{2\theta}-\frac{1}{2^*_s}\right)\|u\|^{2\theta}\\
&\quad-\left(\frac{1}{2\theta}-\frac{1}{2^*_s}\right)^{-\frac{1-\gamma}{2\theta-1+\gamma}}\left[\lambda\left(\frac{1}{1-\gamma}+\frac{1}{2^*_s}\right)|\Omega|^{\frac{2^*_s-1+\gamma}{2^*_s}}S^{-\frac{1-\gamma}{2}}\right]^{\frac{2\theta}{2\theta-1+\gamma}}\\
&\geq\left(\frac{1}{2\theta}-\frac{1}{2^*_s}\right)S^{\frac{2^*_s\theta}{2^*_s-2\theta}} -\left(\frac{1}{2\theta}-\frac{1}{2^*_s}\right)^{-\frac{1-\gamma}{2\theta-1+\gamma}}\left[\lambda\left(\frac{1}{1-\gamma}+\frac{1}{2^*_s}\right)|\Omega|^{\frac{2^*_s-1+\gamma}{2^*_s}}S^{-\frac{1-\gamma}{2}}\right]^{\frac{2\theta}{2\theta-1+\gamma}}
\end{alignedat}
$$
which contradicts \eqref{livello}, since \eqref{costanti}. This concludes the proof.
\end{proof}

We now give a control from above for functional $J_{n,\lambda}$, under a suitable situation.
For this, we assume, without loss of generality, that $0\in\Omega$.
By \cite{CT} we know that the infimum in \eqref{S} is attained at the function
\begin{equation}\label{ueps}
u_\varepsilon(x)=\frac{\varepsilon^{(N-2s)/2}}{(\varepsilon^2+|x|^2)^{(N-2s)/2}},\quad\mbox{with }\varepsilon>0,
\end{equation}
that is, it holds true that
$$\displaystyle\iint_{\mathbb{R}^{2N}}\frac{|u_\varepsilon(x)-u_\varepsilon(y)|^2}{|x-y|^{N+2s}}dxdy=S\|u_\varepsilon\|^2_{L^{2^*_s}(\mathbb R^N)}.$$
Let us fix $r>0$ such that $B_{4r}\subset\Omega$, where $B_{4r}=\{x\in\mathbb R^N:\,|x|<4r\}$, and let us introduce a cut--off function $\phi\in C^\infty(\mathbb R^N,[0,1])$ such that
\begin{equation}\label{phi}
\phi=\begin{cases}
1 & \mbox{ in }B_r,\\
0 & \mbox{ in }\mathbb R^N\setminus B_{2r}.
\end{cases}
\end{equation}
For any $\varepsilon>0$, we set
\begin{equation}\label{psi}
\psi_\varepsilon=\frac{\phi u_\varepsilon}{\|\phi u_\varepsilon\|^2_{2^*_s}}\in X_0.
\end{equation}
Then, we can prove the following result.

\begin{lemma}\label{lemma3.3} There exist $\psi\in X_0$ and $\lambda_1>0$ such that for any $\lambda\in(0,\lambda_1)$ 
$$\sup_{t\geq0} J_{n,\lambda}(t\psi)<D_1-D_2\lambda^{2\theta/(2\theta-1+\gamma)},$$
with $D_1$, $D_2>0$ given as in \eqref{costanti}.
\end{lemma}
\begin{proof}
Let $\lambda$, $\varepsilon>0$.
Let $u_\varepsilon$ and $\psi_\varepsilon$ be as respectively in \eqref{ueps} and in \eqref{psi}.
By \eqref{jn}, we have $J_{n,\lambda}(t\psi_\varepsilon)\to-\infty$ as $t\to\infty$, so that there exists $t_\varepsilon>0$ such that $J_{n,\lambda}(t_\varepsilon \psi_\varepsilon)=\max_{t\geq0}J_{n,\lambda}(t \psi_\varepsilon)$. By Lemma \ref{mp2} we know that $J_{n,\lambda}(t_\varepsilon \psi_\varepsilon)\geq\alpha>0$. Hence, by continuity of $J_{n,\lambda}$, there exist two numbers $t_0$, $t^*>0$ such that $t_0\leq t_\varepsilon\leq t^*$.

Now, since $\|u_\varepsilon\|_{L^{2^*_s}(\mathbb R^N)}$ is independent of $\varepsilon$, by \cite[Proposition 21]{SV2} we have
$$\|\psi_\varepsilon\|^2\leq\frac{\displaystyle\iint_{\mathbb{R}^{2N}}\frac{|u_\varepsilon(x)-u_\varepsilon(y)|^2}{|x-y|^{N+2s}}dxdy}{\|\phi u_\varepsilon\|^2_{2^*_s}}=S+O(\varepsilon^{N-2s}),$$
from which, by the elementary inequality
$$
(a+b)^p\leq a^p+p(a+1)^{p-1}b,\quad\mbox{for any }a>0,\,b\in[0,1],\,p\geq1,
$$
with $p=2\theta$, it follows that, as $\varepsilon\to0^+$
$$
\|\psi_\varepsilon\|^{2\theta}\leq S^\theta+O(\varepsilon^{N-2s}).
$$
Hence, by the last inequality, \eqref{jn} and being $t_0\leq t_\varepsilon\leq t^*$, for any $\varepsilon>0$ sufficiently small, we have
\begin{equation}\label{binlin1}
J_{n,\lambda}(t_\varepsilon\psi_\varepsilon)\leq\frac{t^{2\theta}}{2\theta}S^\theta+C_1\varepsilon^{N-2s}-\frac{\lambda}{1-\gamma}\int_\Omega[(t_0\psi_\varepsilon+1/n)^{1-\gamma}-(1/n)^{1-\gamma}]dx -\frac{t^{2^*_s}}{2^*_s},
\end{equation}
with $C_1$ a suitable positive constant. 
We observe that
$$
\max_{t\geq0}\left(\frac{t^{2\theta}}{2\theta}S^\theta-\frac{t^{2^*_s}}{2^*_s}\right)
=\left(\frac{1}{2\theta}-\frac{1}{2^*_s}\right)S^{\frac{2^*_s\theta}{2^*_s-2\theta}}.
$$
Thus, by \eqref{binlin1} it follows that
\begin{equation}\label{binlin2}
J_{n,\lambda}(t_\varepsilon\psi_\varepsilon)\leq\left(\frac{1}{2\theta}-\frac{1}{2^*_s}\right)S^{\frac{2^*_s\theta}{2^*_s-2\theta}}+C_1\varepsilon^{N-2s}-\frac{\lambda}{1-\gamma}\int_\Omega[(t_0\psi_\varepsilon+1/n)^{1-\gamma}-(1/n)^{1-\gamma}]dx.
\end{equation}
Now, let us consider a positive number $q$ satisfying
\begin{equation}\label{q}
\frac{(N-2s)(1-\gamma)-2q(N-2s)(1-\gamma)+2^*_sqN}{2^*_s}\cdot\frac{2\theta}{2\theta-1+\gamma}\cdot\frac{1}{N-2s}+1-\frac{2\theta}{2\theta-1+\gamma}<0,
\end{equation}
that is, being $2<2\theta<2^*_s$, $N>2s$ and $\gamma\in(0,1)$, such that
$$
0<q<\frac{(N-2s)(2^*_s-2\theta)(1-\gamma)}{2\theta N(2^*_s-2)+4\theta N\gamma+8\theta s(1-\gamma)}.
$$
By the elementary inequality
$$
a^{1-\gamma}-(a+b)^{1-\gamma}\leq -(1-\gamma)b^{\frac{1-\gamma}{p}}a^{\frac{(p-1)(1-\gamma)}{p}},\quad\mbox{for any }a>0,\,b>0\mbox{ large enough},\,p>1,
$$
with $p=2^*_s/2$, and considering $\varepsilon<r^{1/q}$ sufficiently small, with $r$ given by \eqref{phi}, we have
\begin{equation}\label{binlin3}
\begin{alignedat}4
-\frac{1}{1-\gamma}\int_{\{x\in\Omega:\,|x|\leq\varepsilon^q\}}&[(t_0\psi_\varepsilon+1/n)^{1-\gamma}-(1/n)^{1-\gamma}]dx\\
&\leq-\widetilde{C}\varepsilon^{\frac{(N-2s)(1-\gamma)}{2^*_s}}\int_{\{x\in\Omega:\,|x|\leq\varepsilon^q\}}\left[
\frac{1}{(|x|^2+\varepsilon^2)^{\frac{N-2s}{2}}}\right]^{\frac{2(1-\gamma)}{2^*_s}}dx\\
&\leq-C_2\varepsilon^{\frac{(N-2s)(1-\gamma)-2q(N-2s)(1-\gamma)+2^*_s qN}{2^*_s}},
\end{alignedat}
\end{equation}
with $\widetilde{C}$ and $C_2$ two positive constant independent by $\varepsilon$.
By combining \eqref{binlin2} with \eqref{binlin3}, we get
\begin{equation}\label{binlin4}
J_{n,\lambda}(t_\varepsilon\psi_\varepsilon)\leq\left(\frac{1}{2\theta}-\frac{1}{2^*_s}\right)S^{\frac{2^*_s\theta}{2^*_s-2\theta}}+C_1\varepsilon^{N-2s}-C_2\lambda\varepsilon^{\frac{(N-2s)(1-\gamma)-2q(N-2s)(1-\gamma)+2^*_s qN}{2^*_s}}.
\end{equation}
Thus, let us consider $\lambda^*>0$ such that
$$
D_1-D_2\lambda^{\frac{2\theta}{2\theta-1+\gamma}}>0\quad\mbox{for any }\lambda\in(0,\lambda^*)
$$
and let us set 
$$
\begin{alignedat}4
\nu_1=\displaystyle\frac{2q\theta}{(2\theta-1+\gamma)(N-2s)},\quad &\nu_2=\displaystyle\frac{2\theta[(N-2s)(1-\gamma)-2q(N-2s)(1-\gamma)+2^*_s qN]}{2^*_s(2\theta-1+\gamma)(N-2s)}+1,\\
\nu_3=\displaystyle\nu_2-\frac{2\theta}{2\theta-1+\gamma},\quad&\lambda_1=\displaystyle\min\left\{\lambda^*,r^{1/\nu_1},
\left(\frac{C_2}{C_1+D_2}\right)^{-1/\nu_3}\right\},
\end{alignedat}
$$
where $r$ and $q$ are given respectively in \eqref{phi} and \eqref{q}, while we still consider $D_1$ and $D_2$ as defined in \eqref{costanti}.
Then, by considering $\varepsilon=\lambda^{\nu_1/q}$ in \eqref{binlin4}, since \eqref{q} implies that $\nu_3<0$, for any $\lambda\in(0,\lambda_1)$ we have
$$
J_{n,\lambda}(t_\varepsilon\psi_\varepsilon)\leq D_1+C_1\lambda^{\frac{2\theta}{2\theta-1+\gamma}}-C_2\lambda^{\nu_2}
=D_1+\lambda^{\frac{2\theta}{2\theta-1+\gamma}}(C_1-C_2\lambda^{\nu_3})
<D_1-D_2\lambda^{\frac{2\theta}{2\theta-1+\gamma}},
$$
which concludes the proof.
\end{proof}

We can now prove the existence result for \eqref{Pn} by applying the mountain pass theorem.
\begin{theorem}\label{seconda} There exists $\overline{\lambda}>0$ such that, for any $\lambda\in(0,\overline{\lambda})$, problem \eqref{Pn} has a positive solution $v_n\in X_0$, with 
\begin{equation}\label{necessario}
\alpha<J_{n,\lambda}(v_n)<D_1-D_2\lambda^{2\theta/(2\theta-1+\gamma)},
\end{equation}
where $\alpha$, $D_1$ and $D_2$ are given respectively in Lemma \ref{mp} and \eqref{costanti}.
\end{theorem}
\begin{proof}
Let $\overline{\lambda}=\min\{\lambda_0,\lambda_1\}$, with $\lambda_0$ and $\lambda_1$ given respectively in Lemmas \ref{mp} and \ref{lemma3.3}.
Let us consider $\lambda\in(0,\overline{\lambda})$.
By Lemma \ref{mp2}, functional $J_{n,\lambda}$ verifies the mountain pass geometry. For this, we can set the critical mountain pass level as
\begin{equation}\label{mm}
c_{n,\lambda}=\inf_{g\in\Gamma}\max_{t\in[0,1]}J_{n,\lambda}(g(t)),
\end{equation}
where
$$\Gamma=\left\{g\in C([0,1],X_0):\,\,g(0)=0,\; J_{n,\lambda}(g(1))<0\right\}.$$
By Lemmas \ref{mp2} and \ref{lemma3.3}, we get
$$
0<\alpha<c_{n,\lambda}\leq\sup_{t\geq0} J_{n,\lambda}(t\psi)<D_1-D_2\lambda^{2\theta/(2\theta-1+\gamma)}.
$$
Hence, by Lemma \ref{palais} functional $J_{n,\lambda}$ satisfies the Palais--Smale condition at level $c_{n,\lambda}$. Thus, the mountain pass theorem gives the existence of a critical point $v_n\in X_0$ for $J_{n,\lambda}$ at level $c_{n,\lambda}$. Since $J_{n,\lambda}(v_n)=c_{n,\lambda}>\alpha>0=J_{n,\lambda}(0)$, $v_n$ is a nontrivial solution of \eqref{Pn}. Furthermore, by \eqref{weak2} with test function $\varphi=v^-_n$ and inequality \eqref{eur}, we can see that $\|v^-_n\|=0$, which implies $v_n$ is nonnegative. By the maximum principle in \cite[Proposition 2.2.8]{S}, we have that $v_n$ is a positive solution of \eqref{Pn}, concluding the proof.
\end{proof}
\section{A second solution for problem \eqref{P}}\label{sec finale}
In this last section we prove the existence of a second solution for problem \eqref{P}, as a limit of solutions of the perturbed problem \eqref{Pn}. For this, here we need the assumption that $\partial\Omega$ is continuous, in order to apply a density argument for space $X_0$.
\begin{proof}[\textit Proof of Theorem \ref{main}]
Let us consider $\overline{\lambda}$ as given in Theorem \ref{seconda} and let $\lambda\in(0,\overline{\lambda})$.
Since $\overline{\lambda}\leq\lambda_0$, by Theorem \ref{prima} we know that problem \eqref{P} admits a solution $u_0$ with $J_\lambda(u_0)<0$.

In order to find a second solution for \eqref{P}, let $\{v_n\}_n$ be a family of positive solutions of \eqref{Pn}.
By \eqref{S}, \eqref{subadd}, \eqref{necessario} and the H\"older inequality, we have
$$
\begin{aligned}
D_1-D_2\lambda^{\frac{2\theta}{2\theta-1+\gamma}}&> J_{n,\lambda}(v_n)-\frac{1}{2^*_s}\langle J'_{n,\lambda}(v_n), v_n\rangle\\
&=\left(\frac{1}{2\theta}-\frac{1}{2^*_s}\right)\|v_n\|^{2\theta}-\frac{\lambda}{1-\gamma}\int_\Omega[(v_n+1/n)^{1-\gamma}-(1/n)^{1-\gamma}]dx\\
&\quad+\frac{\lambda}{2^*_s}\int_\Omega(v_n+1/n)^{-\gamma}v_n dx\\
&\geq\left(\frac{1}{2\theta}-\frac{1}{2^*_s}\right)\|v_n\|^{2\theta}-\frac{\lambda}{1-\gamma}\int_\Omega v_n^{1-\gamma}dx\\
&\geq\left(\frac{1}{2\theta}-\frac{1}{2^*_s}\right)\|v_n\|^{2\theta}-\frac{\lambda}{1-\gamma}|\Omega|^{\frac{2^*_s-1+\gamma}{2^*_s}}S^{-\frac{1-\gamma}{2}}\|v_n\|^{1-\gamma},
\end{aligned}
$$
which yields that $\{v_n\}_n$ is bounded in $X_0$, being $1-\gamma<1<2\theta$.
Hence, by using \cite[Lemma 8]{SV} and \cite[Theorem 4.9]{B}, there exist a subsequence, still
denoted by $\{v_n\}_n$, and a function $v_0\in X_0$
such that
\begin{equation}\label{2.4}
\begin{array}{ll}
v_n\rightharpoonup v_0\text{ in }X_0,\quad & \|v_n\|\rightarrow \mu, \\
v_n\rightharpoonup v_0\mbox{ in } L^{2^*_s}(\Omega),\quad &\left\|v_n-v_0\right\|_{2^*_s}\to\ell,\\
v_n\rightarrow v_0\mbox{ in } L^p(\Omega)\text{ for any }p\in[1,2^*_s),\quad&v_n\to v_0\text{ a.e.  in }\Omega.
\end{array}
\end{equation}
We want to prove that $v_n\to v_0$ in $X_0$ as $n\to\infty$. When $\mu=0$, by \eqref{2.4} we have $v_n\to 0$ in $X_0$ as $n\to\infty$. For this, we suppose $\mu>0$.
We observe that
$$0\leq\frac{v_n}{(v_n+1/n)^\gamma}\leq v_n^{1-\gamma}\quad\mbox{a.e. in }\Omega,
$$
so by the Vitali convergence theorem and \eqref{2.4}, it follows that
\begin{equation}\label{2.5}
\lim_{n\to\infty}\int_{\Omega}\frac{v_n}{(v_n+1/n)^\gamma}dx=\int_{\Omega}v_0^{1-\gamma}dx.
\end{equation}
By using \eqref{weak2} for $v_n$ and test function $\varphi=v_n$, by \eqref{2.4} and \eqref{2.5}, as $n\to\infty$ we have
\begin{equation}\label{4.2}
\mu^{2\theta}-\lambda \int_\Omega v_0^{1-\gamma}dx+\|v_n\|^{2^*_s}_{2^*_s}=o(1).
\end{equation}
For any $n\in\mathbb N$, by an immediate calculation in \eqref{Pn} we see that
$$\|v_n\|^{2\theta}(-\Delta)^sv_n\geq\min\left\{1,\frac{\lambda}{2^\gamma}\right\}\quad\mbox{in }\Omega.
$$
Thus, since $\{v_n\}_n$ is bounded in $X_0$ and by using a standard comparison argument (see \cite[Lemma 2.1]{BMP}) and the maximum principle in \cite[Proposition 2.2.8]{S}, for any $\widetilde{\Omega}\subset\subset\Omega$, there exists a constant $c_{\widetilde{\Omega}}>0$ such that
\begin{equation}\label{basso}
v_n\geq c_{\widetilde{\Omega}}>0,\quad\mbox{a.e. in }\widetilde{\Omega}\mbox{ and for any }n\in\mathbb N.
\end{equation}
Now, let $\varphi\in C^\infty_0(\Omega)$, with $\mbox{\small supp }\varphi=\widetilde{\Omega}\subset\subset\Omega$. By \eqref{basso} we have
$$0\leq\left|\frac{\varphi}{(v_n+1/n)^\gamma}\right|\leq\frac{|\varphi|}{c_{\widetilde{\Omega}}^\gamma}\quad\mbox{a.e. in }\Omega,
$$
so that by \eqref{2.4} and the dominated convergence theorem
\begin{equation}\label{2.6}
\lim_{n\to\infty}\int_{\Omega}\frac{\varphi}{(v_n+1/n)^\gamma}dx=\int_{\Omega}v_0^{-\gamma}\varphi dx.
\end{equation}
Thus, by considering \eqref{weak2} for $v_n$, sending $n\to\infty$ and using \eqref{2.4} and \eqref{2.6}, for any $\varphi\in C^\infty_0(\Omega)$ it follows that
\begin{equation}\label{4.3}
\mu^{2(\theta-1)}\langle v_0,\varphi\rangle-\lambda \int_\Omega v_0^{-\gamma}\varphi dx
 +\int_\Omega v_0^{2^*_s-1}\varphi dx=0.
\end{equation}
However, since $\partial\Omega$ is continuous, by \cite[Theorem 6]{FSV} the space $C^\infty_0(\Omega)$ is dense in $X_0$. Thus, by a strandard density argument, \eqref{4.3} holds true for any $\varphi\in X_0$.
By combining \eqref{4.2} and \eqref{4.3} with test function $\varphi=v_0$, as $n\to\infty$ we get
$$
\mu^{2(\theta-1)}(\mu^2-\|v_0\|^2)=\|v_n\|^{2^*_s}_{2^*_s}-\|v_0\|^{2^*_s}_{2^*_s}+o(1)
$$
and by \eqref{2.4} and \cite[Theorem 2]{BL} we have
\begin{equation}\label{crucial}
\mu^{2(\theta-1)}\lim_{n\to\infty}\|v_n-v_0\|^2=\ell^{2^*_s}.
\end{equation}
If $\ell=0$, then $v_n\to v_0$ in $X_0$ as $n\to\infty$, since $\mu>0$.

Let us suppose $\ell>0$ by contradiction.
Arguing as in Lemma \ref{palais}, by \eqref{2.4} and \eqref{crucial} we get \eqref{eq31}.
Therefore, being $\theta\geq1$, by \eqref{subadd}, \eqref{eq31}, \eqref{necessario}, \eqref{2.4}, the H\"older inequality and the Young inequality, we have 
$$
\begin{alignedat}4
D_1-D_2\lambda^{\frac{2\theta}{2\theta-1+\gamma}}&>J_{n,\lambda}(v_n)-\frac{1}{2^*_s}\left\langle J'_{n,\lambda}(v_n),v_n\right\rangle\\
&\geq\left(\frac{1}{2\theta}-\frac{1}{2^*_s}\right)\left(\mu^{2\theta}+\|v_0\|^{2\theta}\right)-\lambda\left(\frac{1}{1-\gamma}+\frac{1}{2^*_s}\right)|\Omega|^{\frac{2^*_s-1+\gamma}{2^*_s}}S^{-\frac{1-\gamma}{2}}\|v_0\|^{1-\gamma}\\
&\geq\left(\frac{1}{2\theta}-\frac{1}{2^*_s}\right)S^{\frac{2^*_s\theta}{2^*_s-2\theta}}\\
&\quad-\left(\frac{1}{2\theta}-\frac{1}{2^*_s}\right)^{-\frac{1-\gamma}{2\theta-1+\gamma}}\left[\lambda\left(\frac{1}{1-\gamma}+\frac{1}{2^*_s}\right)|\Omega|^{\frac{2^*_s-1+\gamma}{2^*_s}}S^{-\frac{1-\gamma}{2}}\right]^{\frac{2\theta}{2\theta-1+\gamma}}
\end{alignedat}
$$
which is the desired contradiction, thanks to \eqref{costanti}.

Therefore, $v_n\to v_0$ in $X_0$ as $n\to\infty$ and by \eqref{weak} and \eqref{weak2} we immediately see that $v_0$ is a solution of problem \eqref{P+}. Furthermore, by \eqref{necessario} we have $J_\lambda(v_0)\geq\alpha>0$, which also implies that $v_0$ is nontrivial. Reasoning as at the end of the proof of Theorem \ref{seconda}, we conclude that $v_0$ is a positive solution of \eqref{P+} and so $v_0$ also solves problem \eqref{P}.
Finally, $v_0$ is different from $u_0$, since $J_\lambda(v_0)>0>J_\lambda(u_0)$.
\end{proof}

\section*{Acknowledgments}
The author is supported by {\em Coordena\c c\~ao de Aperfei\c conamento de pessoal de n\'ivel superior} through the fellowship PNPD--CAPES 33003017003P5.
The author is member of the {\em Gruppo Nazionale per l'Analisi Ma\-tema\-tica, la Probabilit\`a e
le loro Applicazioni} (GNAMPA) of the {\em Istituto Nazionale di Alta Matematica ``G. Severi"} (INdAM).

\end{document}